\documentclass[11pt,twoside, final]{amsart}
\copyrightinfo{0}{Iranian Mathematical Society}
\pagespan{1}{\pageref*{LastPage}}
\usepackage{etoolbox,lastpage}
\commby{}
\date{\scriptsize   Received: , Accepted: .}
\usepackage{amsmath,amsthm,amscd,amsfonts,amssymb,enumerate}
\usepackage{subfigure}
\usepackage{graphicx}
\usepackage{color}
\usepackage[colorlinks]{hyperref}

\newtheorem{theorem}{Theorem}[section]

\newtheorem{lemma}[theorem]{Lemma}
\newtheorem{procedure}[theorem]{Procedure}
\newtheorem{corollary}[theorem]{Corollary}
\theoremstyle{definition}
\newtheorem{definition}[theorem]{Definition}
\newtheorem{example}[theorem]{Example}
\theoremstyle{remark}
\newtheorem{remark}[theorem]{Remark}
\numberwithin{equation}{section}
\def \Xint#1{\mathchoice
   {\XXint\displaystyle\textstyle{#1}}%
   {\XXint\textstyle\scriptstyle{#1}}%
   {\XXint\scriptstyle\scriptscriptstyle{#1}}%
   {\XXint\scriptscriptstyle\scriptscriptstyle{#1}}%
   \!\int}
\def \XXint#1#2#3{{\setbox0=\hbox{$#1{#2#3}{\int}$}
     \vcenter{\hbox{$#2#3$}}\kern-.5\wd0}}

\def \dashint{\Xint-}
\begin{document}

%% The title of the paper goes here.  Edit your title.

\title[Approximation for the Extrema's Distributions]{A Weak Approximation for the Extrema's Distributions of L\'evy Processes}
%% Now edit the following to give First Author name and address:
%% $^*$ for the corresponding author.

\author[Amir T. Payandeh Najafabadi]{Amir T. Payandeh Najafabadi$^*$}
\address[Amir T. Payandeh Najafabadi]{Department of Mathematical Sciences, Shahid Beheshti
University, G.C. Evin, 1983963113, Tehran, Iran.}
\email{amirtpayandeh@sbu.ac.ir}

\author[Dan Z. Kucerovsky]{Dan Z. Kucerovsky}
\address[Dan Z. Kucerovsky]{Department of Mathematics and Statistics, University of New
Brunswick, Fredericton, N.B. Canada E3B 5A3.}
\email{dkucerov@unb.ca}
%% If there are three of more authors they are added in the obvious way.

  \thanks{$^*$Corresponding author}
%------------------------------------------------------------------------------------%
%%
%% Use the following command to make the title for the paper.
%
 %\CoverPage

 \maketitle
%
%%% The following environment is needed for the abstract.
%%%

\begin{abstract}
Suppose $X_{t}$ is a one-dimensional and real-valued L\'evy
process started from $X_0=0$, which ({\bf 1}) its nonnegative
jumps measure $\nu$ satisfying $\int_{\Bbb
R}\min\{1,x^2\}\nu(dx)<\infty$ and ({\bf 2}) its stopping time
$\tau(q)$ is \emph{either} a geometric \emph{or} an exponential
distribution with parameter $q$ independent of $X_t$ and
$\tau(0)=\infty.$ This article employs the Wiener-Hopf
Factorization (WHF) to find, an $L^{p^*}({\Bbb R})$ (where
$1/{p^*}+1/p=1$ and $1<p\leq2$), approximation for the extrema's
distributions of $X_{t}.$ Approximating the finite (infinite)-time
ruin probability as a direct application of our findings has been
given. Estimation bounds, for such approximation method, along
with two approximation procedures and
several examples are explored.
\end{abstract}

\section{\bf Introduction}
Suppose that $X_t$ is a one-dimensional and real-valued L\'evy
process started from $X_0=0$ and defined by a triple
$(\mu,\sigma,\nu):$ the drift $\mu\in{\Bbb R},$ the volatility
$\sigma\geq0,$ and the jumps measure $\nu$ which is given by a
nonnegative function defined on ${\Bbb R}\setminus\{0\}$
satisfying $\int_{\Bbb R}\min\{1,x^2\}\nu(dx)<\infty.$ Moreover,
suppose that the stopping time $\tau(q)$ is \emph{either} a
geometric \emph{or} an exponential distribution with parameter $q$
independent of the L\'evy process $X_t$ and $\tau(0)=\infty.$ The
L\'evy-Khintchine formula states that the characteristic exponent
$\psi$ (i.e., $\psi(\omega)=\ln (E(\exp(i\omega
X_1))),~\omega\in{\Bbb R}$) can be represented by
\begin{eqnarray}
\label{Levy-Khintchine} \psi(\omega) &=&
i\mu\omega-\frac{1}{2}\sigma^2\omega^2\\ &&\nonumber+\int_{{\Bbb
R}}(e^{i\omega x}-1-i\omega
xI_{[-1,1]}(x))\nu(dx),~~\omega\in{\Bbb R}.
\end{eqnarray}
The extrema of the L\'evy process $X_t$ are given by
\begin{eqnarray*}
 M_q = \sup\{X_s:~s\leq\tau(q)\} &\& &  I_q
=\inf\{X_s:~s\leq\tau(q)\}.
\end{eqnarray*}
The Wiener-Hopf Factorization (WHF) is a well known technique to
study the characteristic functions of the extrema random variables
(see \cite{Bertoin}. Namely, the WHF states that: ({\bf i})
product of their characteristic functions equal to the
characteristic function of L\'evy process $X_t$ at its stopping
time $\tau(q),$ say $X_{\tau(q)}$ and ({\bf ii}) random variable
$M_q$ ($I_q$) is infinitely divisible, positive (negative), and
has zero drift.

In the cases that, the characteristic function of L\'evy process
$X_t$ either a rational function or can be decomposed as a product
of two sectionally analytic functions in the closed upper, i.e.,
${\Bbb C}^+:=\{\lambda:~\lambda\in{\Bbb
C}~\hbox{and}~\Im(\lambda)\geq0\},$ and lower half complex planes,
i.e., ${\Bbb C}^-:=\{\lambda:~\lambda\in{\Bbb
C}~\hbox{and}~\Im(\lambda)\leq0\}.$ Then, the characteristic
functions of random variables $M_q$ and $I_q$ can be determined
explicitly (see \cite{Payandeh-Kucerovsky2011}). \cite{LM2005}
considered a L\'evy process $X_t$ which its negative jumps is
distributed according to a mixture-gamma family of distributions
and its positive jumps measure has an arbitrary distribution. They
established that the characteristic function of such a L\'evy
process can be decomposed as a product of a rational function in
an arbitrary function, which are analytic in ${\Bbb C}^+$ and
${\Bbb C}^-,$ respectively. They also provided an analog result
for a L\'evy process whose its corresponding positive jumps
measure follows from a mixture-gamma family of distributions while
its negative jumps measure is an arbitrary one, more details can
be found in \cite{LM2008}.

Unfortunately, in the most situations, the characteristic function
of the process {\it neither} is a rational function {\it nor} can
be decomposed as a product of two sectionally analytic functions
in ${\Bbb C}^+$ and ${\Bbb C}^-.$ Therefore, the characteristic
functions of $M_q$ and $I_q$ should be expressed in terms of a
Sokhotskyi-Plemelj integral (see Equation,
\ref{Plemelj-integral}). But, this form, also, presents some
difficulties in numerical work due to slow evaluation and
numerical problems caused by singularities near the integral
contour (see \cite{KP}). To overcome these difficulties, an
appropriate (in some sense) approximation method has to be
considered. It is well known that a L\'evy process $X_t$ which its
jumps distribution follows from the phase-type distribution has a
rational characteristic function (see \cite{Doney}).
\cite{Kuznetsov2010} utilized this fact and approximated a jumps
measure $\nu$ of a ten-parameter L\'evy processes (named
$\beta-$family of L\'evy process) by a sequence of the phase-type
measures. Then, he determined the characteristic functions of
random variables $M_q$ and $I_q,$ approximately.
\cite{Kuznetsov2012} extended \cite{Kuznetsov2010}'s findings to
class of Meromorophic L\'evy processes. Moreover, \cite{Kwasnicki}
provided a uniform approximation for the cumulative distribution
function of $M_{\tau(q)}$ whenever $X_t$ is a symmetric L\'evy
process. \cite{KPS} employed the Shannon sampling method to find
the distributions of the extrema for a wide class of L\'evy
processes.

This article begins with an extension of  \cite{KP}'s results for
the multiplicative WHF
\begin{eqnarray}\label{multiplicative-WHF}
  \Phi^+(\omega)\Phi^-(\omega)&=& g(\omega)~~~\omega\in{\Bbb R},
\end{eqnarray}
where $g(\cdot)$ is a given function with some certain conditions
(see below) and $\Phi^\pm(\cdot)$ are to be determined. Then, it
utilizes such results to approximate the extrema's distributions
of a class of L\'evy processes. Estimation bounds, for such
approximate method, along with two approximation procedures are
given.

Section 2 collects some useful elements for other sections.
Moreover, it provides an $L^p({\Bbb R}),~1<p\leq2$ approximation
technique for solving a multiplicative WHF
\eqref{multiplicative-WHF}. Section 3 considers the problem of
approximating the extrema's density functions for a class of
L\'evy processes. Then, it develops two approximate techniques for
situations where those density functions cannot be determined,
explicitly. Error bounds for such techniques are given. Several
examples are given in Sections 4. Section 5 provides concluding
remarks along with some suggestions for other application of our
techniques.
\section{Preliminaries}
The \emph{Sokhotskyi-Plemelj integral} for $s(\cdot),$ which
satisfies the H\"older condition, is defined by a principal value
integral, as follows
\begin{eqnarray}
\label{Plemelj-integral}\phi_s(\lambda ):=\frac 1{2\pi
i}\dashint_{ {\Bbb R}}\frac {s(x)}{x -\lambda}dx ,~~\hbox{for}~
\lambda\in {\Bbb C}.\end{eqnarray} It is worth mentioning that,
the Sokhotskyi-Plemelj integral can be existed for non-integrable
function such as $sin(x).$ Therefore, the Sokhotskyi-Plemelj
integral $\phi_s(\cdot)$ should be viewed different from the usual
integral over ${\Bbb R}$.

The radial limits of the Sokhotskyi-Plemelj integral of
$s(\cdot),$ are given by $\phi^{\pm}_s(\omega
)=\displaystyle\lim_{\lambda\rightarrow \omega
+i0^{\pm}}\phi_s(\lambda )$ and satisfy the following {\it jump
formulas:} ({\bf 1}) $\phi^{\pm}_s(\omega )=\pm s(\omega
)/2+\phi_s(\omega ),$ for $\omega\in {\Bbb R}$ and ({\bf 2})
$\phi^{\pm}_s(\omega )=\pm s(\omega )/2+H_s(\omega )/(2i),$ where
$H_s(\omega)$ stands for the \emph{Hilbert transform} of
$s(\cdot)$ and $\omega\in {\Bbb R}.$

The multiplicative WHF is the problem of finding an analytic and
bounded, except on the real line, function $\Phi(\cdot)$ where its
upper and lower radial limits $\Phi^\pm(\cdot)$ satisfy Equation
\eqref{multiplicative-WHF}. Given function $g(\cdot)$ is a bounded
above by 1, zero index\footnote{The index of a complex-valued
function $f$ on a smooth oriented curve $\Gamma,$ such that
$f(\Gamma)$ is closed and compact, is defined to be the winding
number of $f(\Gamma)$ about the origin (see
\cite{Payandeh-Kucerovsky2007}, \S 1), for more technical
details.}, continuous, and positive function which satisfies the
H\"older condition on ${\Bbb R},$ $g(0)=1,$ and $g(\omega)\neq0$
for all $\omega\in{\Bbb R}.$

The following extends \cite{KP}'s results to the multiplicative
WHF \eqref{multiplicative-WHF}. We begin with what we term the
Resolvent Equation for Sokhotskyi-Plemelj integrals.
\begin{lemma} The {\it Sokhotskyi-Plemelj} integral of a function
$f(\cdot)$ satisfies $\phi_f (\lambda) - \phi_f (\mu) =
(\lambda-\mu)\phi_{\frac{f(x)}{x-\lambda}}(\mu),$ where $\lambda$
and $\mu$ are real or complex values. \label{lem:resolvent}
\end{lemma}
\begin{proof} In general, $(x-\lambda)^{-1} - (x-\mu)^{-1} = (\lambda-\mu)(x-\mu)^{-1}(x-\lambda)^{-1}.$
Then, see \cite{DunfordSchwartz}, we have an equation of Cauchy
integrals, where $\Gamma=\Bbb R$:
$$\frac{1}{2\pi i} \int_\Gamma \frac{f(x)}{x-\lambda}dx - \frac{1}{2\pi i} \int_\Gamma \frac{f(x)}{x-\mu}dx = \frac{\lambda-\mu}{2\pi i} \int_\Gamma \frac{f(x)}{(x-\mu)(x-\lambda)}dx .$$

The above is valid only for $\lambda$ and $\mu$ not on the real
line. However, by Equation \eqref{Plemelj-integral} the values of
$\phi_f(\cdot)$ on the real line are obtained by averaging the
limit from above, $\phi^+_f(\cdot)$, and the limit from below,
$\phi^{-}_f.$ We thus obtain the stated equation in all
cases.\end{proof}
\begin{lemma}
\label{Solution-WH-For-Our-Paper} Suppose $\Phi^{\pm}(\cdot)$ are
sectionally analytic functions that satisfy the multiplicative WHF
\eqref{multiplicative-WHF}. Moveover, suppose that given function
$g(\cdot)$ is a zero index function which satisfies the H\"older
condition and $g(0)=1.$\footnote{The condition $g(0)=1$ does not
always hold in the multiplicative WHF, but happen to arise in our
application, and can lead to complications. Lemma
\eqref{Solution-WH-For-Our-Paper} is used to simplifying this
case.} Then $ \Phi^\pm(\lambda)=\exp\{\pm(\phi_{\ln
g}(\lambda)-\phi_{\ln g}(0))\},$ where $\phi_{\ln g}(\cdot)$
stands for the Sokhotskyi-Plemelj integral of $\ln g(\cdot).$
\end{lemma}
\textbf{Proof.}  Using the \cite{Gakhov}'s suggestion for solving
the homogeneous WHF \eqref{multiplicative-WHF} gives, see also
\cite{LM2008}:
\begin{eqnarray*}
  \Phi^\pm(\lambda) &=& \exp\{\pm\frac{\lambda}{2\pi i}\dashint_{ {\Bbb R}}\frac {\ln g(x)/x}{x
  -\lambda}dx\}.
\end{eqnarray*}
Lemma \eqref{lem:resolvent} with $f\equiv\ln g$ gives
 $\phi_{\ln g} (\lambda) - \phi_{\ln g} (\mu) = (\lambda-\mu)\phi_{\frac{\ln g(x)}{x-\lambda}}(\mu).$
Letting $\lambda$ goes to zero from the above, in the complex
plane, and using the fact that $\ln g(0) = 0$, Equation
\eqref{Plemelj-integral} lets us to conclude that
 $\phi_{\ln g} (0) - \phi_{\ln g} (\mu) = -\mu\phi_{\frac{\ln g(x)}{x}}(\mu).$
 Substituting this into the above equation for $\Phi^\pm(\cdot)$ gives our claimed result.
~$\square$

Using the jump formula one can conclude that
\begin{eqnarray}\label{solutions-of-WH-in-term-g}
\Phi^\pm(\omega)&=&\sqrt{g(\omega)}\exp \{\pm\frac{i}{2}(H_{\ln
g}(0)-H_{\ln g}(\omega))\},
\end{eqnarray}
where $H_{\ln g}(\cdot)$ stands for the Hilbert transform of $\ln
g(\cdot).$

The Carlemann's method explores a situation which one may evaluate
solutions of the multiplicative WHF \eqref{multiplicative-WHF}
directly, rather than using the Sokhotskyi-Plemelj integrations.
The Carlemann's method states that: if $g(\cdot)$ can be
decomposed as a product of two sectionally analytic functions
$g^+(\cdot)$ and $g^-(\cdot),$ respectively in ${\Bbb C}^+$ and
${\Bbb C}^-.$ Then, solutions of the multiplicative WHF
\eqref{multiplicative-WHF} are given by $\Phi^+\equiv g^+$ and
$\Phi^-\equiv g^-.$

In a situation that $g(\cdot)$ is a rational function
$\frac{P(x)}{Q(x)}$ that has no poles or zeros on ${\Bbb R}.$
Using the Carlemann's method, we may conclude that the
multiplicative WHF problem can be solved by factoring the
polynomial $P$ ($Q$), and then let $P^+$ ($Q^+$) be the product of
those factors of $P$ ($Q$) that have zeros in ${\Bbb C}^-$, and
$P^-$ ($Q^-$) be the product of those factors that have zeros in
${\Bbb C}^+$. Then, setting $g^{+}(x)=\frac{P^{+}(x)}{Q^{+}(x)}$
and $g^{-}(x)=\frac{P^{-}(x)}{Q^{-}(x)}$ gives us (up to a scalar
multiple) our desired factorization.

The \emph{Hausdorff-Young} theorem (see \cite{Pandey}) states
that: if $s(\cdot)$ is an $L^p({\Bbb R})$ function. Then,
$s(\cdot)$ and its corresponding the Fourier transform, say $\hat
{s}(\cdot),$ satisfy $ ||\hat {s}||_{p^*}\leq (2\pi
)^{-1/p}||s||_p$, where $1\leq p\leq2$ and $1/p+1/p^*=1.$ From the
Hausdorff-Young Theorem, one can observe that if $\{s_n(\cdot) \}$
is a sequence of functions converging in $L^p({\Bbb R}),$ $1\leq
p\leq2,$ to $s(\cdot).$  Then, the Fourier transforms of the
$s_n(\cdot)$ converges in $L^{p^*}({\Bbb R}),$ to the Fourier
transform of $s(\cdot)$, where $1/p+1/p^*=1.$ The converse is
false.

A similar property for the Hilbert transform is well known as the
Titmarsh-Riesz lemma (see \cite{Pandey}). The Titmarsh-Riesz lemma
says that: if $s(\cdot)$ is an $L^p({\Bbb R})$ function, where $1<
p\leq2.$ Then, $||H_s||_p\leq tan(\pi/(2p))||s||_p,$ where
$H_s(\cdot)$ stands for the Hilbert transform of $s(\cdot).$ Using
the Titmarsh-Riesz lemma, one may conclude that if
$\{f_n(\cdot)\},$ is a sequence of functions which converge, in
$L^p({\Bbb R}),~1<p\leq2,$ to $f(\cdot).$ Then, the Hilbert
transforms $H(f_n)$ converge, in $L^p({\Bbb R}),~1<p\leq2,$ to the
Hilbert transform of $f(\cdot).$

The well known Paley-Wiener theorem states that: if $F(\cdot)$ is
a function in $L^2({\Bbb R}).$ Then, the real-valued function
$F(\cdot)$ vanishes on ${\Bbb R}^{-}$ if and only if the Fourier
transform $F(\cdot)$, say, $\hat {F}(\cdot)$ is holomorphic on
${\Bbb C}^{+}$ and the $L^2({\Bbb R})$-norm of the functions
$x\mapsto\hat {F}(x+iy_0)$ are uniformly bounded for all $y_0\geq
0.$

The following, from \cite{KP}, recalls some further useful
properties of functions in $L^p({\Bbb R}),$ for $1<p\leq2,$ space.
\begin{lemma}
\label{measure-of-error} Suppose $s(\cdot)$ and $r(\cdot)$ are two
$L^p({\Bbb R}),~1<p\leq2,$ functions.   Then,
\begin{description}
     \item[i)] $||\sqrt{s}-\sqrt{r}||_p\leq
     \frac{1}{2\sqrt{a}}||s-r||_p,$ whenever both $s(\cdot)$ and $r(\cdot)$ are bounded, above by $a,$ functions;
     \item[ii)] $||\ln s-\ln r||_p\leq a^{-1}||s-r||_p,$ whenever both $s(\cdot)$ and $r(\cdot)$
     are positive-valued and bounded, above by $a,$ functions ;
     \item[iii)] $||e^{-is/2}-e^{-ir/2}||_p\leq\frac{1}{2}||s-r||_p,$ whenever $s(\cdot)$
     and $r(\cdot)$ are real-valued functions.
\end{description}
\end{lemma}
In many situations, WHF \eqref{multiplicative-WHF} cannot be
solved explicitly and has to be solved approximately (see
\cite{KP}). The following develops an approximation technique to
solve a multiplicative WHF \eqref{multiplicative-WHF}.
\begin{theorem}
\label{approximation-WH} Suppose $\Phi^{\pm}(\cdot)$ are two
sectionally analytic functions satisfying the multiplicative WHF
\eqref{multiplicative-WHF} where
\begin{description}
    \item[$A_1$)] $g(\cdot)$ is real, positive,  bounded above by a, index zero, satisfies the H\"older condition, and
$g(0)=1;$
    \item[$A_2$)] There exist a sequence of functions $g_n(\cdot)$ where converge, in $L^p({\Bbb R}),$ $1<p\leq2,$ to $g(\cdot)$.
\end{description}
Then, $\Phi^{\pm}(\cdot)$ can be approximated by
$\Phi^{\pm}_n(\cdot),$ where
\begin{eqnarray*}
  ||\Phi^\pm_n-\Phi^\pm||_p &\leq&
  \frac{1}{2}tan(\frac{\pi}{2p})||g_n-g||^2_p+(tan(\frac{\pi}{2p})+\frac{1}{2})||g_n-g||_p.
\end{eqnarray*}
\end{theorem}
\textbf{Proof.} Set $k(\omega):=-H_{\ln g}(\omega)+H_{\ln g}(0)$
and $k_n(\omega):=-H_{\ln g_n}(\omega)+H_{\ln g_n}(0).$ Now, from
Equation \eqref{solutions-of-WH-in-term-g} and Lemma
\eqref{measure-of-error} observe that $||\Phi^\pm_n-\Phi^\pm||_p$
\begin{eqnarray*}
  ~ &=&
  ||\sqrt{g_n}e^{\pm ik_n/2}-\sqrt{g}e^{\pm ik/2}||_p\\
  &\leq&
  \left[||\sqrt{g_n}-\sqrt{g}||_p+||\sqrt{g}||_p\right]||e^{\pm ik_n/2}-e^{ik/2}||_p+|e^{\pm ik/2}|||\sqrt{g_n}-\sqrt{g}||_p\\
  &\leq&\frac{1}{2}\left[||\sqrt{g_n}-\sqrt{g}||_p+||\sqrt{g}||_p\right]||-H_{\ln g_n}(\omega)+H_{\ln g_n}(0)+H_{\ln g}(\omega)-H_{\ln g}(0)||_p\\
  &&+|e^{\pm ik/2}|||\sqrt{g_n}-\sqrt{g}||_p\\
  &\leq& \left[||\sqrt{g_n}-\sqrt{g}||_p+||\sqrt{g}||_p\right]||H_{\ln g_n}-H_{\ln
  g}||_p+||\sqrt{g_n}-\sqrt{g}||_p\\&&\hbox{since $k$ and $k_n$ are real-valued
  functions}\\
  &\leq& \left[||\sqrt{g_n}-\sqrt{g}||_p+||\sqrt{g}||_p\right]tan(\frac{\pi}{2p})||\ln(g_n)-\ln(
  g)||_p+||\sqrt{g_n}-\sqrt{g}||_p\\
  &\leq& tan(\frac{\pi}{2p})\left[\frac{1}{2}||g_n-g||_p+1\right]||g_n-
  g||_p+\frac{1}{2}||g_n-g||_p\\
  &=& \frac{1}{2}tan(\frac{\pi}{2p})||g_n-g||^2_p+(tan(\frac{\pi}{2p})+\frac{1}{2})||g_n-g||_p.~\square
\end{eqnarray*}
Now, we recall definition of the positive-definite function which
plays a vital roles in the rest of this article.
\begin{definition}
A positive-definite function is a complex-valued function $f:{\Bbb
R}\rightarrow {\Bbb C}$ such that for any real numbers
$x_1,\cdots, x_n$ the $n\times n$ square matrix
$A=(f(x_i-x_j))_{i,j=1}^n$ is a positive semi-definite matrix.
\end{definition}
In the theory of the Fourier transform, it is well known that
``$f(\cdot)$ is a continuous positive-definite function on ${\Bbb
R}$ if and only if its corresponding the Fourier transform is a
(positive) measure'', see \cite{Bochner1959} for more details.
\begin{lemma}
\label{h_1_h_2_pdfs} Suppose $\phi:{\Bbb R}\rightarrow {\Bbb C}$
is a positive-definite function which two equations
 $q_1-\phi(\omega)= 0$ and
$1-q_2\exp\{-\phi(\omega)\}= 0$ have not any solution on ${\Bbb
R},$ where $q_1>0$ and $q_2\in(0,1).$ Then,
$h_1(\omega)=q_1/(q_1-\phi(\omega))$ and
$h_2(\omega)=(1-q_2)/(1-q_2\exp\{-\phi(\omega)\})$ are
positive-definite functions.
\end{lemma}
\textbf{Proof.} Using the Taylor expansion of $q_1/(q_1-x)$ and
$(1-q_2)/(1-q_2\exp\{-x\}),$ about zero, one may restated
$h_1(\cdot)$ and $h_2(\cdot)$ as
\begin{eqnarray*}
  h_1(\omega) &=& \sum_{k=1}^{\infty}\frac{\phi^k(\omega)}{q_1^k} \\
  h_2(\omega) &=&
  1+\frac{q_2}{q_2-1}\phi(\omega)+\frac{q_2(q_2+1)}{2(q_2-1)^2}\phi^2(\omega)+\frac{q_2(q_2^2+4q_2+1)}{6(q_2-1)^3}\phi^3(\omega)+\cdots
\end{eqnarray*}
Now, the desired proof arrives from the fact that the product of
two positive-definite functions is again a positive-definite
function (see \cite{Schoenberg}). $\square$

Now, we provide two classes of positive-definite rational
functions which play a vital role in numerical section of this
article.
\begin{lemma}
 \label{Lemma_positive-def_rational}
Consider the following two class of rational functions
$\mathcal{D}$ and $\mathcal{D}^*.$
\begin{eqnarray*}
\nonumber\mathcal{D}:&=&\{r(\omega);~
r(\omega)=A_0+\sum_{k=1}^{n}\sum_{j=1}^{m_k}\sum_{l=1}^{4}C_{kj}r^j_{lk}(\omega);~A_0~\&~C_{kj}\geq0\};\\
\mathcal{D}^\star:&=&\{r(\omega);~
r(\omega)=A_0+\sum_{k=1}^{n}\sum_{l=1}^{2}C_{kj}r_{lk}(\omega);~A_0~\&~C_{kj}\geq0\},
\end{eqnarray*}
where
\begin{eqnarray*}
 r_{1k}(\omega) &=& \frac{1}{i\omega+\beta_k}~(\hbox{where~}\beta_k>0);\\
 r_{2k}(\omega) &=& \frac{1}{-i\omega+\beta_k}~(\hbox{where~}\beta_k>0);\\
 r_{3k}(\omega)
 &=&\frac{1}{(i\omega+\beta_k)(i\omega+\beta_k+\alpha_ki)(i\omega+\beta_k-\alpha_ki)}~(\hbox{where~}\alpha_k,~\beta_k>0);\\
 r_{4k}(\omega)
 &=&\frac{1}{(-i\omega+\beta_k)(-i\omega+\beta_k+\alpha_ki)(-i\omega+\beta_k-\alpha_ki)}~(\hbox{where~}\alpha_k,~\beta_k>0).
\end{eqnarray*}
Then, ({\bf i}) the Fourier transform of functions in
$\mathcal{D}$ are nonnegative and real-valued functions. ({\bf
ii}) the Fourier transform of functions in $\mathcal{D}^*$ are
nonnegative, real-valued, and completely monotone functions.
\end{lemma}
\textbf{Proof.} Nonnegativity of the Fourier transform of
functions in $\mathcal{D}$ (or $\mathcal{D}^*$) arrives from the
fact that $r_{lk}$ (for $l=1,\cdots,4$) and their powers are
positive-definite rational functions. Now, from the Bernstein's
theorem observe that a real-valued function defined on ${\Bbb
R}^+$ is a completely monotone function, whenever it is a mixture
of exponential functions, see \cite{Schilling} for more details.
$\square$

The following may be concluded from properties of the WHF given by
\cite{Bertoin}.
\begin{lemma}
\label{Bertoin_1996_properties_WHF} Suppose $g(\cdot)$ in the
multiplicative WHF \eqref{multiplicative-WHF} is a
positive-definite function. Then, solutions, of the multiplicative
WHF \eqref{multiplicative-WHF}, $\Phi^\pm(\cdot)$ are two
positive-definite functions.
\end{lemma}
\textbf{Proof.} First observe that, using the multiplicative WHF
\eqref{multiplicative-WHF} the characteristic function of L\'evy
process $X_t$ at its stopping time $\tau(q),$ say $X_{\tau(q)},$
can be decomposed as a product of the characteristic functions of
two random variables $I_q$ and $M_q,$ see \cite{Bertoin} for more
details. Moreover, \cite{Bochner1955}'s theorem states that ``An
arbitrary function $\phi: {\Bbb R}^n\rightarrow{\Bbb C}$ is the
characteristic function of some random variable if and only if
$\phi(\cdot)$ is positive-definite, continuous at the origin, and
if $\phi(0) = 1$''. The desired proof arrives using these
observations. $\square$

One may readily observe that the characteristic function of the
mixed gamma family of distributions (given below) are belong to
$\mathcal{D}.$
\begin{definition} (Mixed gamma family of distributions)
\label{mixed gamma} A nonnegative random variable $X$ is said to
be distributed according to a mixed gamma distribution if its
density function is given by
\begin{eqnarray*}
  p(x) &=&
  \sum_{k=1}^{\nu}\sum_{j=1}^{n_\nu}c_{kj}\frac{\alpha_k^jx^{j-1}}{(j-1)!}e^{-\alpha_k
  x}I_{[0,\infty)}(x)+\sum_{k=1}^{\nu^*}\sum_{j=1}^{n_{\nu^*}}c^*_{kj}\frac{\beta_k^j(-x)^{j-1}}{(j-1)!}e^{\beta_k
  x}I_{(-\infty,0]}(x)
\end{eqnarray*}
where $c_{k_j}$ and $\alpha_k$ are positive value which satisfy
$\sum_{k=1}^{\nu}\sum_{j=1}^{n_\nu}c_{k_j}=1.$
\end{definition}
We now from \cite{Bracewell} recall some useful properties of the
characteristic function, which plays an important role for the
next sections.
\begin{lemma}
\label{properties-characteristic-function} Suppose ${\hat
p}(\cdot)$ stands for the characteristic function of a
distribution. Then,
\begin{description}
    \item[i)] ${\hat p}(\cdot)$ is a positive-definite function;
    \item[ii)] ${\hat p}(\cdot)$ is a positive-definite rational function whenever its characteristic function belongs to $\mathcal{D}$ given
    by Lemma \eqref{Lemma_positive-def_rational};
    \item[iii)] ${\hat p}(0)=1;$ and the norm of ${\hat p}(\cdot)$ is bounded by
    1.
\end{description}
\end{lemma}
The next section provides an application of Theorem
\eqref{approximation-WH} to the problem of finding the
distributions of the extrema of L\'evy process $X_t,$
approximately.
\section{Main results}
The following lemma restates the characteristic function of L\'evy
process $X_t$ at its stopping time $\tau(q),$ say $X_{\tau(q)}.$
\begin{lemma}
\label{ch_EX_GEO} Suppose $X_{\tau(q)}$ represents L\'evy process
$X_t$ at its stopping time $\tau(q).$ Then, the characteristic
function of $X_{\tau(q)}$ can be restated as:
\begin{description}
    \item[(i)] $q/(q-\psi(\omega)),$ for an exponential stopping
    time $\tau(q)$ with parameter $q>0;$
    \item[(ii)] $(1-q)/(1-q\exp\{-\psi(\omega)\}),$ for a geometric stopping
    time $\tau(q)$ with parameter $q\in(0,1).$
\end{description}
\end{lemma}
\textbf{Proof.} Conditioning on stopping time $\tau(q),$ one may
restates he characteristic function of $X_{\tau(q)}$ as:
\begin{description}
    \item[For part (i)]
\begin{eqnarray*}
  E(e^{i\omega X_{\tau(q)}}) &=& \int_0^\infty
  E(e^{i\omega X_{\tau(q)}}|\tau(q)=t)f_{\tau(q)}(t)dt
 = \int_0^\infty E(e^{i\omega X_{t}})qe^{-qt}dt\\&=& \int_0^\infty e^{\psi(\omega)t}qe^{-qt}dt=\frac{q}{q-\psi(\omega)};
\end{eqnarray*}
    \item[For part (ii)]
\begin{eqnarray*}
  E(e^{i\omega X_{\tau(q)}}) &=& \sum_{n=0}^\infty
  E(e^{i\omega X_{\tau(q)}}|\tau(q)=n)P(\tau(q)=n)= \sum_{n=0}^\infty E(e^{i\omega X_{n}})(1-q)q^n\\
 &=& \sum_{n=0}^\infty e^{\psi(\omega)n}(1-q)q^n=\frac{1-q}{1-q\exp\{-\psi(\omega)\}},
\end{eqnarray*}
\end{description}
where for both cases, the second equality arrives from the fact
that $X_{t}$ and $\tau(q)$ are independent and the third equality
obtains from definition of the characteristic exponent $\psi$ and
infinitely divisibility of L\'evy process $X_t.$ $\square$

The following theorem represents an error bound for approximating
the density functions of extrema of a L\'evy process.
\begin{theorem}
\label{Approximate-distributions} Suppose $X_{t}$ is a L\'evy
process defined by a triple $(\mu,\sigma,\nu).$ Moreover, suppose
that:
\begin{description}
    \item[$A_1$)] The stopping time $\tau(q)$ is either a geometric or an exponential
distribution with parameter $q$ independent of $X_t$ and
$\tau(0)=\infty;$
    \item[$A_2$)] The $r_n(dx)$ are a sequence of positive-definite rational functions which converge, in $L^{p^*}({\Bbb
    R})$ (where $1/{p^*}+1/p=1$ and $1<p\leq2$), to characteristic exponent
    $q/(q-\psi(dx))$ (or $(1-q)/(1-q\exp\{-\psi(dx)\})$ for geometric stoping time)
\end{description}
Then, the density function of the suprema and infima of the L\'evy
process $X_t$, denoted $f_q^+$ and $f_q^-$, respectively, can be
approximated, in $L^{p^*}({\Bbb R})$ sense, by a sequence of the
density functions $f_{q,n}^+$ and $f_{q,n}^-$ where:
\begin{description}
    \item[i)] For exponentially distributed stopping time
    $\tau(q),$ for $q>0,$
    $$||f^\pm_q-f^\pm_{q,n}||_{p^*}\leq\frac{1}{2}tan(\frac{\pi}{2p^*})||r_n-\frac{q}{q-\psi}||^2_{p^*}+(tan(\frac{\pi}{2p^*})+\frac{1}{2})||r_n-\frac{q}{q-\psi}||_{p^*};$$
    \item[ii)] For geometric stopping time $\tau(q),$ for $q\in(0,1),$
   $$||f^\pm_q-f^\pm_{q,n}||_{p^*}\leq\frac{1}{2}tan(\frac{\pi}{2p^*})||r_n-\frac{1-q}{1-qe^{-\psi}}||^2_{p^*}+(tan(\frac{\pi}{2p^*})+\frac{1}{2})||r_n-\frac{1-q}{1-qe^{-\psi}}||_{p^*}.$$
\end{description}
\end{theorem}
\textbf{Proof.} From \cite{Bertoin} and Lemma \eqref{ch_EX_GEO},
one can observe that the Fourier transform of the density
functions of random variables $M_q$ and $I_q,$ say respectively
$\Phi^+$ and $\Phi^-,$ satisfy \emph{either} the multiplicative
WHF $\Phi^+(\omega)\Phi^-(\omega)=q/(q-\psi(\omega)),$ where
$\omega\in{\Bbb R}$ (for exponentially distributed stopping time)
\emph{or} the multiplicative WHF
$\Phi^+(\omega)\Phi^-(\omega)=(1-q)/(1-q\exp\{-\psi(\omega)\}),$
where $\omega\in{\Bbb R}$ (for geometric stopping time). Now, from
the fact that the expressions $q/(q-\psi(\cdot))$ and
$(1-q)/(1-q\exp\{-\psi(\cdot)\})$ are the characteristic function
of the L\'evy process $X_t$, at exponential and geometric stopping
time, respectively, we observe that both expressions are bounded
above by 1 because of the property of the characteristic function
given by Lemma (\ref{properties-characteristic-function}, part
ii).  For part (i), from Theorem \eqref{approximation-WH} observe
that
\begin{eqnarray*}
  ||\Phi^\pm_n-\Phi^\pm||_{p}
  &\leq&\frac{1}{2}tan(\frac{\pi}{2p})||r_n-\frac{q}{q-\psi}||^2_{p}+(tan(\frac{\pi}{2p})+\frac{1}{2})||r_n-\frac{q}{q-\psi}||_{p}.
\end{eqnarray*}
The rest of proof arrives from an application of the
Hausdorff-Young Theorem. The proof of part (ii) is quite similar.
 $\square$
\begin{remark}
\label{atom_zero} In case that the distribution of $I_q$ or $M_q$
has an atom at $x=0.$ Then, it corresponding probability mass
function at zero can be found, approximately, by
\begin{eqnarray*}
  P(I_q=0) = \lim_{\omega\rightarrow\infty}\Phi^-(-i\omega) & \& &  P(M_q=0) = \lim_{\omega\rightarrow\infty}\Phi^+(i\omega).
\end{eqnarray*}
\end{remark}
Using the fact that the Compound Poisson has bounded
characteristic exponent $\psi(\cdot)$. The following formulates
result of the above theorem in terms of the jumps measure
$\nu(dx).$
\begin{theorem}
\label{Approximate-distributionsCompound-Poisson}(Compound
Poisson) Suppose $X_{t}$ is a Compound Poisson process defined by
a triple $(\mu,\sigma,\nu).$ Moreover, suppose that
\begin{description}
    \item[$A_1$)] the stopping time $\tau(q)$ is either a geometric or an exponential
distribution with parameter $q$ independent of $X_t$ and
$\tau(0)=\infty;$
    \item[$A_2$)] the $\nu_n(dx)$ are a sequence of the density functions which converge in $L^2({\Bbb R}),$ to jumps measure $\nu$
    and $\int_{-1}^{1}x\nu_n(dx)=\int_{-1}^{1}x\nu(dx).$
\end{description}
Then, the density functions of the suprema and infima of the
Compound Poisson process $X_t$, denoted by $f_q^+(\cdot)$ and
$f_q^-(\cdot)$, respectively, can be approximated by a sequence of
the density functions $f_{q,n}^+(\cdot)$ and $f_{q,n}^-(\cdot)$
where:
\begin{description}
    \item[i)] For exponentially distributed stopping time $\tau(q),$
    $$||f^\pm_q-f^\pm_{q,n}||_2\leq\frac{1}{q^2\sqrt{8\pi}}||\nu_n-\nu||^2_2+\frac{3}{2q}||\nu_n-\nu||_2;$$
    \item[ii)] For geometric stopping time $\tau(q),$
   $$||f^\pm_q-f^\pm_{q,n}||_2\leq\frac{(1-q)^2}{q^2\sqrt{8\pi}}||\nu_n-\nu||^2_2+\frac{3(1-q)}{2q}||\nu_n-\nu||_2.$$
\end{description}
\end{theorem}
\textbf{Proof.} Suppose $\psi_n(\cdot)$ are sequence of the
characteristic exponent corresponding to $\nu_n(dx).$ For part (i)
using result of Theorem \eqref{Approximate-distributions}, one may
conclude that
\begin{eqnarray*}
  ||\Phi^\pm_n-\Phi^\pm||_{2} &\leq&\frac{1}{2}||\frac{q}{q-\psi_n}-\frac{q}{q-\psi}||^2_{2}+\frac{3}{2}||\frac{q}{q-\psi_n}-\frac{q}{q-\psi}||_{2}\\
  &\leq& \frac{1}{2q^2}||\psi_n-\psi||^2_{2}+\frac{3}{2q}||\psi_n-\psi||_{2}\\
  &\leq&
  \frac{1}{4\pi q^2}||\nu_n-\nu||^2_{2}+\frac{3}{q\sqrt{8\pi}}||\nu_n-\nu||_{2}.
\end{eqnarray*}
The second inequality arrives from the fact that the
characteristic function $q/(q-\psi(\cdot))$ is bounded above by 1,
while the third inequality comes from the Levy-Khintchine
representation (Equation, \ref{Levy-Khintchine}) along with
conditions $A_2$ and the Hausdorff-Young Theorem. The rest of
proof arrives from an application of the Hausdorff-Young Theorem.
The proof of part (ii) is quite similar.
 $\square$

\section{Application to the finite (infinite)-time ruin probability}
Suppose surplus process of an insurance company can be restated as
\begin{eqnarray}
\label{surplus_process}
  U_t &=& u+X_t,
\end{eqnarray}
where L\'evy process $X_t$ and $u>0$ stands for initial
wealth/reserve of the process.

The finite-time ruin probability for the such surplus process is
denoted by ${\bf R}^{(q)}(u)$ and defined by
\begin{eqnarray*}
  {\bf R}^{(q)}(u) &=& P(T\leq\tau_q|U_0=u),
\end{eqnarray*}
where $T$ is the hitting time, i.e., $T:=\inf\{t:~U_t\leq
0|U_{0}=u\}$ and $\tau_{u}$ is a random stoping time. Such the
stoping time has been distributed corroding to {\it either} an
exponential distribution (with mean $1/q$) {\it or} a geometric
distribution (with mean $(1-q)/q$).

The infinite-time ruin probability for the surplus process
\eqref{surplus_process} is denoted by ${\bf R}(u)$ and defined by
\begin{eqnarray*}
  {\bf R}(u) &=& P(T<\infty|U_0=u).
\end{eqnarray*}
The infinite-time ruin probability ${\bf R}(u),$ also, can be
evaluated by ${\bf R}(u)=\lim_{q\rightarrow\ 0}{\bf R}^{(q)}(u).$

Using Alili \& Kyprianou (2005, Lemma 1 with setting $\beta=0$ and
replacing $X$ by $-X$)'s findings, one may conclude that: in a
situation that the infima density function $f_q^-$ of the L\'evy
process $X_t$ is available, the finite-time ruin probability under
the above surplus process can be restated as
\begin{eqnarray*}
  {\bf R}^{(q)}(u) &=& P(I_q<-u)=\int_{-\infty}^{-u}f_q^-(y)dy.
\end{eqnarray*}

Now using an $L_p({\Bbb R})-$norm for an integral operator (see
Theorem 3.36 in \cite{Cruz-Uribe}), one may restate results of
Theorem \eqref{Approximate-distributions} and Theorem
\eqref{Approximate-distributionsCompound-Poisson} for
approximating the finite-time ruin probability under the surplus
process \eqref{surplus_process} as the following two corollaries.

\begin{corollary}
\label{Approximate-Ruin} Suppose $X_{t}$ in the surplus process
\eqref{surplus_process} is a L\'evy process defined by a triple
$(\mu,\sigma,\nu).$ Moreover, suppose that:
\begin{description}
    \item[$A_1$)] The stopping time $\tau(q)$ is either a geometric or an exponential
distribution with parameter $q$ independent of $X_t$ and
$\tau(0)=\infty;$
    \item[$A_2$)] The $r_n(dx)$ are a sequence of positive-definite rational functions which converge, in $L^{p^*}({\Bbb
    R})$ (where $1/{p^*}+1/p=1$ and $1<p\leq2$), to characteristic exponent
    $q/(q-\psi(dx))$ (or $(1-q)/(1-q\exp\{-\psi(dx)\})$ for geometric stoping time)
\end{description}
Then, the finite-time ruin probability under the surplus process
\eqref{surplus_process}, say ${\bf R}^{(q)}(u),$ can be
approximated, in $L^{p^*}({\Bbb R})$ sense, by a sequence of the
ruin probability, say ${\bf R}^{(q)}_{n}(u),$ where:
\begin{description}
    \item[i)] For exponentially distributed stopping time
    $\tau(q),$ for $q>0,$
    $$||{\bf R}^{(q)}-{\bf R}^{(q)}_{n}||_{p^*}\leq\frac{1}{2}tan(\frac{\pi}{2p^*})||r_n-\frac{q}{q-\psi}||^2_{p^*}+(tan(\frac{\pi}{2p^*})+\frac{1}{2})||r_n-\frac{q}{q-\psi}||_{p^*};$$
    \item[ii)] For geometric stopping time $\tau(q),$ for $q\in(0,1),$
   $$||{\bf R}^{(q)}-{\bf R}^{(q)}_{n}||_{p^*}\leq\frac{1}{2}tan(\frac{\pi}{2p^*})||r_n-\frac{1-q}{1-qe^{-\psi}}||^2_{p^*}+(tan(\frac{\pi}{2p^*})+\frac{1}{2})||r_n-\frac{1-q}{1-qe^{-\psi}}||_{p^*}.$$
\end{description}
\end{corollary}

\begin{corollary}
\label{Approximate-Ruin-Compound-Poisson}(Compound Poisson)
Suppose $X_{t}$ in the surplus process \eqref{surplus_process} is
a Compound Poisson process defined by a triple $(\mu,\sigma,\nu).$
Moreover, suppose that
\begin{description}
    \item[$A_1$)] the stopping time $\tau(q)$ is either a geometric or an exponential
distribution with parameter $q$ independent of $X_t$ and
$\tau(0)=\infty;$
    \item[$A_2$)] the $\nu_n(dx)$ are a sequence of the density functions which converge in $L^2({\Bbb R}),$ to jumps measure $\nu$
    and $\int_{-1}^{1}x\nu_n(dx)=\int_{-1}^{1}x\nu(dx).$
\end{description}
Then, the finite-time ruin probability under the surplus process
\eqref{surplus_process}, say ${\bf R}^{(q)}(u),$ can be
approximated, in $L^{p^*}({\Bbb R})$ sense, by a sequence of the
finite-time ruin probability, say ${\bf R}^{(q)}_{n}(u),$ where:
\begin{description}
    \item[i)] For exponentially distributed stopping time $\tau(q),$
    $$||{\bf R}^{(q)}-{\bf R}^{(q)}_{n}||_2\leq\frac{1}{q^2\sqrt{8\pi}}||\nu_n-\nu||^2_2+\frac{3}{2q}||\nu_n-\nu||_2;$$
    \item[ii)] For geometric stopping time $\tau(q),$
   $$||{\bf R}^{(q)}-{\bf R}^{(q)}_{n}||_2\leq\frac{(1-q)^2}{q^2\sqrt{8\pi}}||\nu_n-\nu||^2_2+\frac{3(1-q)}{2q}||\nu_n-\nu||_2.$$
\end{description}
\end{corollary}
It is worth mentioning that the above results may be obtained for
the infinite-time ruin probability by letting $q\rightarrow 0.$

The next section provides some practical applications of the above
findings.
\section{Examples}
In the first step, this section provides two particle procedures
for the problem of finding the density functions of the suprema
and infima of a L\'evy process.

Using the fact that the characteristic exponent
$\psi(i\omega),~\omega\in{\Bbb R},$ is a real-valued function,
(see \cite{Bertoin}) along with Lemma
\eqref{Bertoin_1996_properties_WHF}, we suggest the following two
procedures to generate approximation density functions for $M_q$
and $I_q.$
\begin{procedure}
\label{procedure-1} Suppose $X_{t}$ is a Meromorophic L\'evy
process\footnote{L\'evy process $X_{t}$ is said to belong to the
Meromorophic class of L\'evy process if and only if
$\bar{\nu}^+(x)=\nu(x,\infty)$ and
$\bar{\nu}^-(x)=\nu(-\infty,-x)$ are two completely monotone
functions and characteristic exponents $\psi(\cdot)$ is a
Meromorophic function, see \cite{Kuznetsov2012} for more details.}
with the characteristic exponents $\psi(\cdot).$ Moreover, suppose
that stopping time $\tau(q)$ is either a geometric or an
exponential distribution with parameter $q$ independent of $X_t$
and $\tau(0)=\infty.$ Then, by the following steps, one can
approximate, in $L^{p^*}({\Bbb R})$ (where $1/{p^*}+1/p=1$ and
$1<p\leq2$) sense, the density functions of the extrema random
variables $M_q$ and $I_q.$
\begin{description}
 \item[Step 1-]
 \begin{description}
  \item[1)] Find out all zeros and poles of $q/(q-\psi(\omega))$
  (or $(1-q)/(1-q\exp\{-\psi(\omega)\})$);
   \item [2)] Define $f^+(\omega)$ as product over all zeros/poles
   lying in ${\Bbb C}^-$ and $f^-(\omega)$ as product over all zeros/poles
   lying in ${\Bbb C}^+;$
    \end{description}
    \item[Step 2)] Determine error of approximating $q/(q-\psi(\omega))$
  (or $(1-q)/(1-q\exp\{-\psi(\omega)\})$) by $f^+(\omega)f^-(\omega);$
    \item[Step 3)] Obtain the density functions of $M_q$
    and $I_q$ by the inverse Fourier transform of $f^+(\cdot)$ and $f^-(\cdot),$
    respectively.
\end{description}
\end{procedure}
\textbf{Proof.} For an exponential stopping time,
\cite{Kuznetsov2012} showed that zeros and poles of
$q/(q-\psi(\omega)),$ respectively, appear as
$\{-i\alpha_n,i\alpha_n\}$ and $\{-i\beta_n,i\beta_n\},$ where
$\cdots<-\beta_1<-\alpha_1<0<\beta_1<\alpha_1<\cdots$
\cite{Kuznetsov2010} proved that $f^+(\omega)f^-(\omega)$ where
\begin{eqnarray*}
  f^+(\omega)= \prod_{n\geq1}\frac{1+i\omega/\alpha_n}{1+i\omega/\beta_n}
  &\& &
  f^-(\omega) = \prod_{n\geq1}\frac{1-i\omega/\alpha_n}{1-i\omega/\beta_n}
\end{eqnarray*}
uniformly approximates $q/(q-\psi(\omega))$. Now observe that, all
terms of $f^+(\cdot)$ and $f^-(\cdot)$ (e.g.
$\frac{1+i\omega/\alpha_n}{1+i\omega/\beta_n}$ or
$\frac{1-i\omega/\alpha_n}{1-i\omega/\beta_n}$) are
positive-definite rational functions. Therefore, $f^+(\cdot)$ and
$f^-(\cdot)$ are two positive-definite rational functions and
analytical in ${\Bbb C}^+$ and ${\Bbb C}^-,$ respectively. An
application of the Paly-Winer theorem warranties that the inverse
Fourier transform of $f^+(\cdot)$ and $f^-(\cdot)$  are two
positive density functions which vanish on ${\Bbb R}^+$ and ${\Bbb
R}^-,$ respectively.

For the geometric stopping time, using the fact that $q<1$ again
one may show that all poles of $(1-q)/(1-q\exp\{-\psi(\cdot)\})$
evaluated by equation $1-q\exp\{-\psi(\omega)\}=0$ or equivalently
by $\ln(q)+\psi(\omega)=0.$ Now, \cite{Kuznetsov2012}'s findings
shows that all poles will be appear as $\{-i\beta_n,i\beta_n\}.$
On the other hands, zeros of $(1-q)/(1-q\exp\{-\psi(\cdot)\})$ are
points where $\psi(\omega)=\infty.$ Therefore, zeros of appear as
$\{-i\alpha_n,i\alpha_n\}.$ The rest of proof is similar.
$\square$

The following examples shows application of the above procedure.
\begin{example}
\label{Stable_processes} Stable processes have been successfully
fitted to stock returns, excess bond returns, foreign exchange
rates, commodity price returns, real estate return data (see,
e.g., \cite{McCulloch} and \cite{Rachev-Mittnik}, financial data
(see, e.g., \cite{Borak2008}), Market- and Credit-Value-at-Risk,
Value-at-Risk, credit risk management (see, e.g., \cite{Rachev}).
With the exception of the normal distribution ($\alpha=2$), stable
distribution are the heavy tailed distributions which paly an
important role in heavy-tail modeling of economic data (see, e.g.,
\cite{Mittnik_Rachev1991} and \cite{Mittnik_Rachev1993}) and
finance data (see, e.g., \cite{Rachev-Mittnik}).

Now consider a symmetric stable process $X_t$ with the homomorphic
characteristic exponent function
$\psi(\omega)=1/(i\mu\omega-\lambda^\alpha|\omega|^\alpha),$ where
$\alpha\in (0,2].$

Using the fact that the real value $\alpha,$ in the above
characteristic exponent, can be constructed from the rational
numbers $m/n,$ where $m$ and $n$ respectively are even and odd
numbers. Now, an expression $q/(q-\psi(\omega^n))$ can be restated
as
\begin{eqnarray*}
  \frac{qi\mu\omega^n-q\lambda^{m/n}\omega^m}{1-qi\mu\omega^n+q\lambda^{m/n}\omega^m} &=&
  (qi\mu\omega^n-q\lambda^{m/n}\omega^m)\prod_{i=1}^{n^+}\frac{1}{\omega-z_i^+}\prod_{i=1}^{n^-}\frac{1}{\omega-z_i^-},
\end{eqnarray*}
where $n^++n^-$ is number of solutions for equation
$1-qi\mu\omega^n+q\lambda^{m/n}\omega^m=0$ in $\omega.$ Moreover,
$z_i^+$ and $z_i^-$ are solutions of the recent equation where
belong to ${\Bbb C}^+$ and ${\Bbb C}^-,$ respectively. Therefore,
approximate solutions for the density function of extrema,
$f_q^\pm,$ are the inverse Fourier transform of
$\phi_n^\pm(\omega):=\sqrt{qi\mu\omega-q\lambda^{m/n}\omega^{m/n}}/\prod_{i=1}^{n^\mp}(\omega^{1/n}-z_i^\mp).$
\end{example}

To implement Procedure \eqref{procedure-1} for the Meromorophic
L\'evy process, one has to determine all zeros and poles of
$q/(q-\psi(\cdot))$ (or $(1-q)/(1-q\exp\{-\psi(\cdot)\})$) which
is a difficult task in may cases. Moreover, in the case where
zeros or poles of $q/(q-\psi(\cdot))$ (or
$(1-q)/(1-q\exp\{-\psi(\cdot)\})$) appear as
$\{\alpha_n\pm\beta_ni\}$ (where at least one of $\alpha_n>0$).
Some terms of decomposition $f^+(\cdot)$ (or $f^-(\cdot)$) are not
positive-definite rational function. Therefore, the inverse
Fourier transform of $f^+(\cdot)$ and $f^-(\cdot)$ can be negative
in some interval. The following procedure extents result of
Procedure \eqref{procedure-1} for such cases and the
non-homomorphic L\'evy processes.

Before stating the second procedure, we need the following lemma.
\begin{lemma}
\label{roots-psi} Suppose $\psi(\cdot)$ stands for the
characteristic exponent of a L\'evy process. Moreover, suppose
that $\alpha_0+\beta_0i$ is a root of $q-\psi(\lambda)=0,$
$\lambda\in{\Bbb C}.$ Then, $-\alpha_0+\beta_0i$ also is root of
$q-\psi(\lambda)=0.$
\end{lemma}
\textbf{Proof.} Using the L\'evy Khintchine formula (Equation,
\eqref{Levy-Khintchine}), equation of $q-\psi(\lambda)=0$ at point
$\alpha_0+\beta_0i$ can be restated as
\begin{eqnarray*}
  \left\{ \begin{gathered}
  -\sigma^2\alpha_0\beta_0i+\alpha_0\mu i+i\int_{{\Bbb R}}\left(e^{-\beta_0}\sin(\alpha_0 x)-\alpha_0xI_{[-1,1]}(x) \right)\nu(dx)=0;\\
  -\frac{1}{2}\sigma^2(\alpha_0^2-\beta_0^2)-\mu\beta_0+\int_{{\Bbb R}}\left(e^{-\beta_0}\cos(\alpha_0 x)-1+\beta_0xI_{[-1,1]}(x)
  \right)\nu(dx)=q.
  \end{gathered}  \right.
\end{eqnarray*}
Since $\sin(\cdot)$ and $\cos(\cdot),$ respectively, are odd and
even functions. Therefore, one may conclude that point
$-\alpha_0+\beta_0i$ satisfies the above system of equations, as
well. $\square$

\begin{procedure}
\label{procedure-2} Suppose $X_{t}$ is a L\'evy process with
characteristic exponents $\psi(\cdot).$ Moreover, suppose that the
stopping time $\tau(q)$ is either a geometric or an exponential
distribution with parameter $q$ independent of $X_t$ and
$\tau(0)=\infty.$ Then, by the following steps, one can
approximate, in $L^{p^*}({\Bbb R})$ (where $1/{p^*}+1/p=1$ and
$1<p\leq2$) sense, the density functions of the extrema random
variables $M_q$ and $I_q.$
\begin{description}
\item[Step 1] Approximating $h(\omega):=q/(q-\psi(\omega)),$ for
the exponential stopping time, (or
$h(\omega):=(1-q)/(1-q\exp\{-\psi(\omega)\}),$ for the geometric
stopping time) by a positive-definite rational function by the
following steps:
\begin{description}
    \item[1)] Find out all poles of $h(\omega)$;
    \item[2)] Based upon such
poles pick up some positive-definite rational functions given in
Lemma \eqref{Lemma_positive-def_rational};
    \item[3)] Approximate $h(\omega)$ by positive-definite rational
function $r(\omega),$ given by Lemma
\eqref{Lemma_positive-def_rational}; \item[4)] Set
$A_0:=\lim_{\omega\rightarrow\infty}h(\omega)$ and $m_k$ equal to
order of $k^{\hbox{th}}$ pole; \item[5)] Determine positive
coefficients $C_{lk}$ by a visual investigation or
\begin{eqnarray*}
  C_{lk} &=& \max\left\{0,argmin\int_{{\Bbb R}}(h(\omega)-r(\omega))^pd\omega\right\}
\end{eqnarray*}
\end{description}
  \item[Step 2)] Determine error of approximating $h(\omega)$ by $r(\omega);$
 \item[Step 3-] Decompose the positive-definite rational function
  $r(\omega)$ as a product of two functions,
  say $f^\pm(\omega),$ which are sectionally analytic and bounded in ${\Bbb
  C}^\pm$;
    \item[Step 4)] Obtain the density functions of $M_q$
    and $I_q$ by the inverse Fourier transform of $f^+(\cdot)$ and $f^-(\cdot),$
    respectively.
\end{description}
\end{procedure}
\textbf{Proof.} Since $h(\omega)$ is a characteristic function.
Lemma \eqref{properties-characteristic-function} warranties that,
it is a positive-definite function and consequently its limit at
infinity, say $A_0,$ is a positive real number. Moreover Lemma
\eqref{roots-psi} warranties that, one may use positive rational
functions $r_{3k}(\cdot)$ and $r_{4k}(\cdot)$ whenever pole with
form $\alpha\pm\beta i$ has been observed. The rest of proof is
similar to Procedure \eqref{procedure-1}. $\square$

\begin{example}
\label{infinite-variation}  Suppose $X_t$ is a L\'evy process with
independent and continuous $\tau(q)$ and a jumps measure
$\nu(dx)=\exp\{\alpha x\}cosech^2(x/2)dx.$ The characteristic
exponent for such L\'evy process is given by
\begin{eqnarray*}
  \psi(\omega) &=&
  -\frac{\sigma^2\omega^2}{2}-i\rho\omega-4\pi(\omega-i\alpha)\coth(\pi(\omega-i\alpha))+4\gamma,
\end{eqnarray*}
where $\gamma=\pi\alpha\cot(\pi\alpha),$
$\rho=4\pi^2\alpha+\frac{4\gamma(\gamma-1)}{\alpha}-\mu,$
$\omega\in{\Bbb R},$ and $\alpha,~\mu,~\hbox{and}~\sigma$ are
given. Note that it is impossible to solve equation
$q-\psi(\omega)=0$ in the general case. Now consider special
cases, whenever $\sigma=\mu=2$ and $\alpha=0.$ Now, we compute the
Wiener-Hopf factorization for $q=5.$ Finding all poles of
$5/(5-\psi(\omega))$ is difficult task. Using Maple 15, one may
readily compute the three first poles as
$\{-0.4781i,~0.5658i,~1.4921i \}.$ On the other hands
$A_0=\lim_{\omega\rightarrow\infty}5/(5-\psi(i\omega))=0.$ Now, we
approximate $5/(5-\psi(\omega))$ by
\begin{eqnarray*}
  r(\omega) &=&
  \frac{C_1}{-i\omega+0.4781}+\frac{C_2}{i\omega+0.5658}+\frac{C_3}{i\omega+1.4921}.
\end{eqnarray*}
A graphical illustration shows that, one may readily chose
$C_1=C_2=C_3=1/4.5$ see Figure 1-a. Error of this approximation is
about $0.08719956902.$ $r(\omega)$ can be restarted as
\begin{eqnarray*}
  r(\omega) &=&
  \frac{(i\omega+0.9560)(-i\omega+1.9123)}{4.5(-i\omega+0.4781)(i\omega+0.5658)(i\omega+1.4921)}\\
  &=&
  \frac{i\omega+0.9560}{\sqrt{4.5}(i\omega+0.5658)(i\omega+1.4921)}\frac{-i\omega+1.9123}{\sqrt{4.5}(-i\omega+0.4781)}\\
  &=& f^-(\omega)f^+(\omega).
\end{eqnarray*}
Therefore, the density function of $I_{\tau(5)}$ and $M_{\tau(5)}$
can be approximated by
\begin{eqnarray*}
  f_{I_{\tau(5)}}(x) &=& 0.5110841035e^{1.4921x}+0.3719983876e^{0.5658x},~\hbox{~~~~~~for}~x\leq0; \\
  f_{M_{\tau(5)}}(x) &=&0.2857404120Dirac(x)+0.4097937547e^{-0.4781x}~\hbox{~for}~x\geq0,
\end{eqnarray*}
where $Dirac(x)$ stands for the dirac delta at point $x=0.$
Figures 1-b and 1-c illustrate behavior of
$f_{I_{\tau(5)}}(\cdot)$ and $f_{M_{\tau(5)}}(\cdot),$
respectively.\end{example}
\begin{center}
\begin{figure}[h!]
\centering\subfigure[]{
\includegraphics[width=4cm,height=4cm]{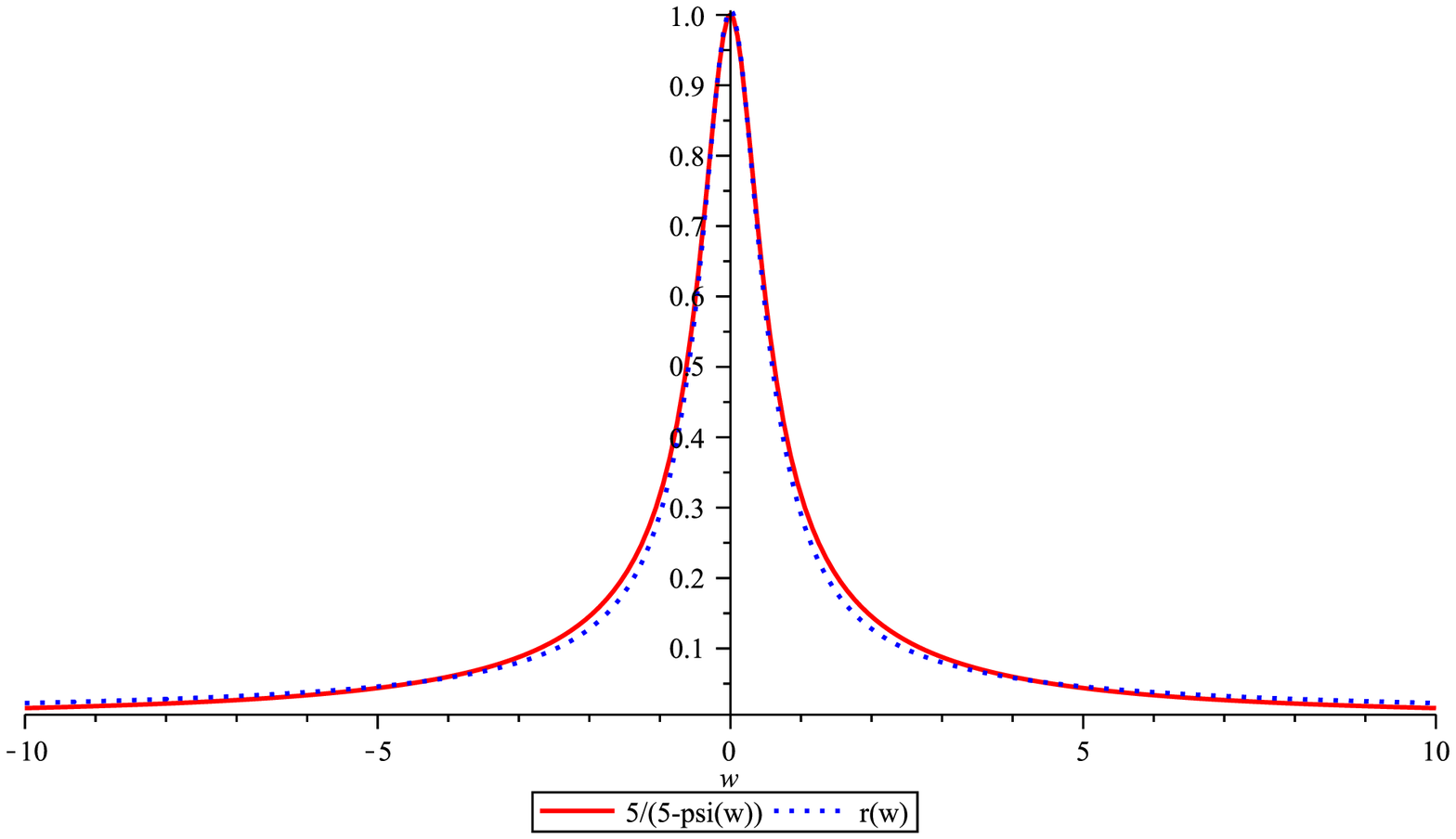}}\subfigure[]{
\includegraphics[width=4cm,height=4cm]{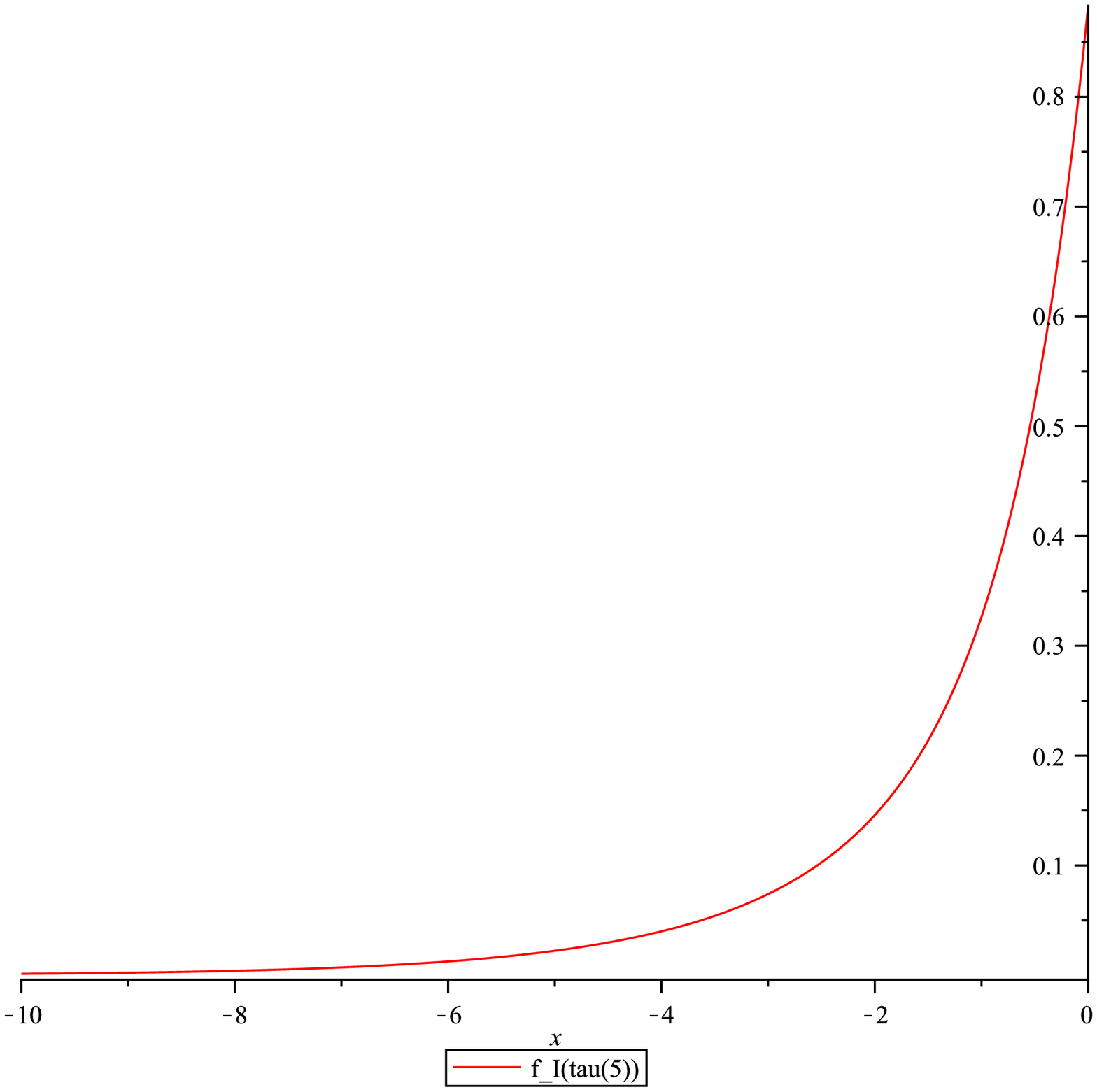}}\subfigure[]{
\includegraphics[width=4cm,height=4cm]{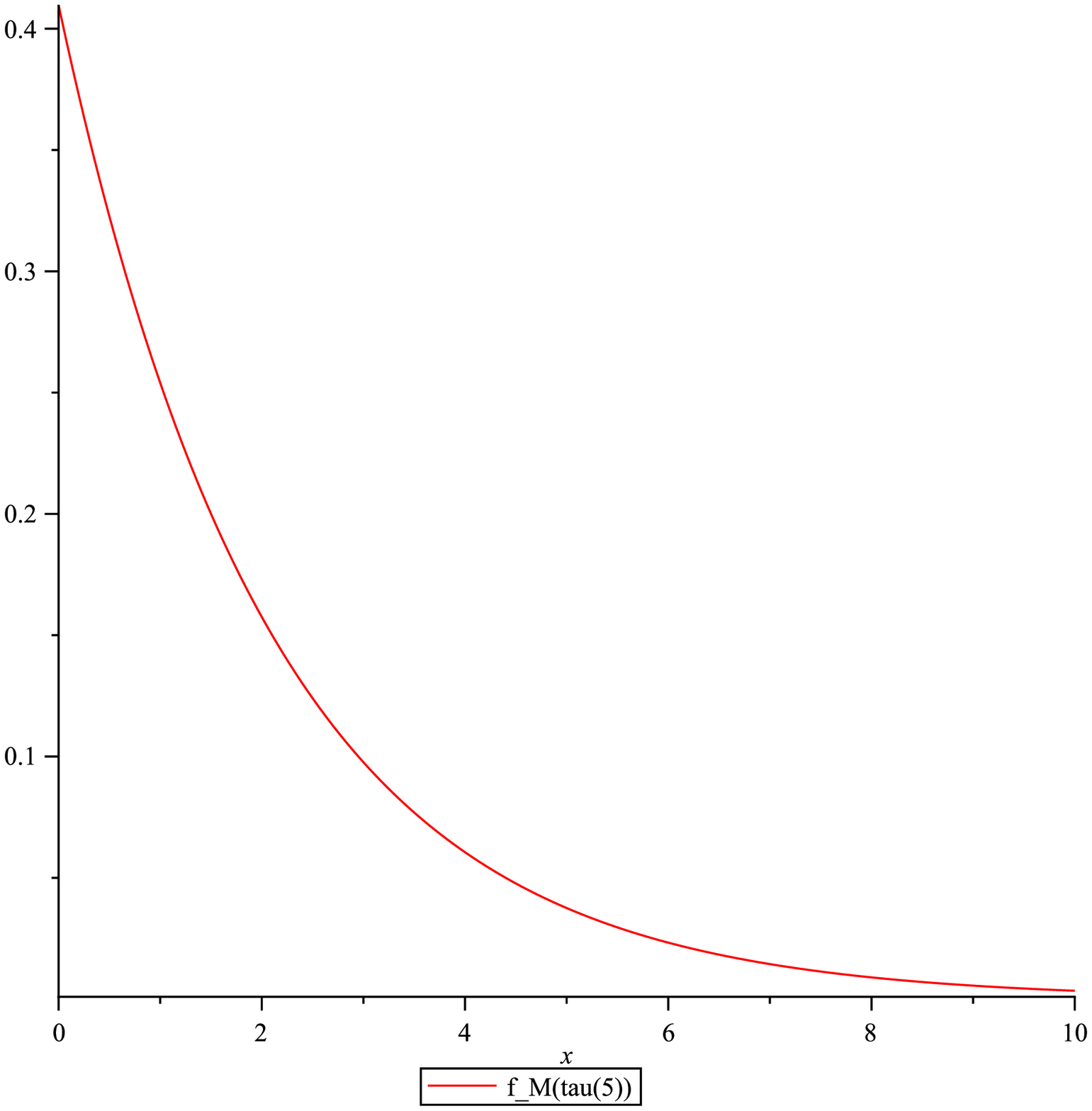}}
\caption{\scriptsize Graphical illustration of: (a)
$\frac{5}{5-\psi(\omega)}$ and its approximation $r(\omega)$; (b)
$f_{I_{\tau(5)}};$ and (c) $f_{M_{\tau(5)}}.$}
\end{figure}
\end{center}

\begin{example}
\label{Ruin-infinite-variation} Suppose $X_t$ in the surplus
process \eqref{surplus_process} is the L\'evy process in Example
\eqref{infinite-variation}. Moreover, suppose that the random
stoping time $\tau(q)$ has an exponential distribution with mean
0.2. Using result of Example \eqref{infinite-variation}, Figure 2
illustrates behavior of the finite-time ruin probability for
different initial value $u.$
\begin{figure}[h!]
\centering
\includegraphics[width=7cm,height=5cm]{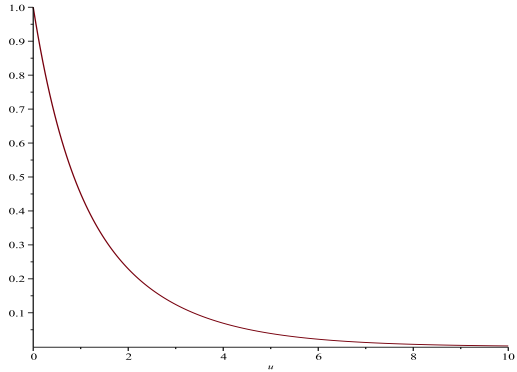}
\caption{\scriptsize Behavior of the finite-time ruin probability
for different initial value $u.$}
\end{figure}
\end{example}
The following example explores situation that roots of
$q-\psi(\omega)=0$ appears in form of $\alpha+i\beta.$
\begin{example}
\label{generalized-hyperbolic-processes} Consider a generalized
hyperbolic process with the characteristic function
$$\phi(\omega)=e^{\psi(\omega)}=e^{i\mu\omega}\left(\frac{\alpha^2-\beta^2}{\alpha^2-(\beta+i\omega)^2} \right)^{\lambda/2}\frac{\mathcal{K}_\lambda(\delta\sqrt{\alpha^2-(\beta+i\omega)^2})}{\mathcal{K}_\lambda(\delta\sqrt{\alpha^2-\beta^2})},$$
where $\lambda,\mu\in{\Bbb R},$ $\alpha,\delta>0,$
$\beta\in(-\alpha,\alpha),$ and $\mathcal{K}_\lambda(\cdot)$ is
the Modified Bessel functions of the third kind with index
$\lambda.$ Many well known processes are member of the class of
generalized hyperbolic L\'evy processes. For $\lambda>0$ and
$\delta\rightarrow 0$ one gets a Variance-Gamma process. The case
$\lambda=-1/2$ corresponds to the normal inverse Gaussian process,
see \cite{Eberlein} for some analytic facts and applications about
the generalized hyperbolic processes. The generalized hyperbolic
process $X_t$ is a pure jump process which can be considered as a
Brownian motion with drift that evolves according to an increasing
Levy process (i.e., subordinator). Such properties make the
generalized hyperbolic process is an appealing process to model
the financial returns, see \cite{Necula} for more details.

Note that it is impossible to solve Equation $q-\psi(\omega)=0$ in
the general case. Now consider special cases, whenever
$\alpha=\mu=2,$ $\beta=-\lambda=1,$ and $\delta=3.$ Now, we
compute the Wiener-Hopf factorization for $q=5.$ Finding all poles
of $5/(5-\psi(\omega))$ is difficult task. Using Maple 15, one may
readily compute the sixth first poles as \footnotesize{$\{\pm
0.4809389066+4.280110446i; \pm 0.9037063690+2.340695867i; \pm
2.516794346+0.4442175550i; \pm 3.756731426-0.9399774855i; \pm
4.853043564-2.318278971i; \pm 5.894258220-3.713000684i; \pm
6.909960755-5.121014155i; \pm 7.912034845-6.538310520i; \pm
8.905992610-7.962083725i; \pm 9.894699100-9.390493630i \}$}.
\normalsize On the other hands
$A_0=\lim_{\omega\rightarrow\infty}5/(5-\psi(i\omega))=0.$

\cite{Rogers} established that the generalized hyperbolic process
has completely monotone jump density. Therefore, One has to
approximate $q/(q-\psi(\omega))$ by function class
$\mathcal{D}^*,$ given by Lemma
\eqref{Lemma_positive-def_rational}. Therefore,
$5/(5-\psi(\omega))$ can be approximated by
\begin{eqnarray*}
  r(\omega) &=&
  \frac{C_1}{i\omega+4.280110446}+\frac{C_2}{i\omega+2.340695867}+\frac{C_3}{i\omega+0.4442175550}\\ &&+\frac{C_4}{-i\omega+ 0.9399774855}
  +\frac{C_5}{-i\omega+2.318278971}+\frac{C_6}{-i\omega+3.713000684}\\ &&+\frac{C_7}{-i\omega+5.121014155}+\frac{C_8}{-i\omega+6.538310520}+\frac{C_9}{-i\omega+7.962083725}\\ &&
  +\frac{C_{10}}{-i\omega+9.390493630}.
\end{eqnarray*}
A graphical illustration shows that, one may readily chose
$C_1=\cdots=C_{10}=0.4,$ see Figure 2-a. An $L_2({\Bbb R})$ error
of this approximation is about \footnotesize{$0.000002527687170,$}
\normalsize which can be improved by choosing more appropriate
coefficients. $r(\omega)$ can be restarted as \footnotesize{
\begin{eqnarray*}
r(\omega) &=&
  \frac{0.4}{i\omega+4.280110446}+\frac{0.4}{i\omega+2.340695867}+\frac{0.4}{i\omega+0.4442175550}\\ &&+\frac{0.4}{-i\omega+ 0.9399774855}+\frac{0.4}{-i\omega+2.318278971}+\frac{0.4}{-i\omega+3.713000684}\\ &&
  +\frac{0.4}{-i\omega+5.121014155}+\frac{0.4}{-i\omega+6.538310520}+\frac{0.4}{-i\omega+7.962083725}\\ &&+\frac{0.4}{-i\omega+9.390493630}\\
 &=&f^-(\omega)f^+(\omega),
\end{eqnarray*}}\normalsize
Therefore, the density function of $I_{\tau(5)}(\cdot)$ and
$M_{\tau(5)}(\cdot)$ can be approximated by
\begin{eqnarray*}
  f_{I_{\tau(5)}}(x) &=&
  0.3268288347e^{0.9399774846x}+0.6308685531e^{9.390493235x}\\
  &&+0.6059253078e^{7.962085726x}+0.4905298019e^{3.713000220x}\\ &&+0.5383620597e^{5.121016047x}
+0.5757366525e^{6.538307461x}\\ &&+0.4259214977e^{2.318279006x}
,~\hbox{~~~~~~for}~x\leq0; \\
  f_{M_{\tau(5)}}(x)
  &=&0.2367700968Dirac(x)+0.4078345184e^{-2.340695867x}\\ &&+0.5390740986e^{-4.280110443x}+0.2582813546e^{-0.4442175554x},\\&&~~~\hbox{~for}~x\geq0.
\end{eqnarray*}

\begin{center}
\begin{figure}[h!]
\centering\subfigure[]{
\includegraphics[width=4cm,height=5cm]{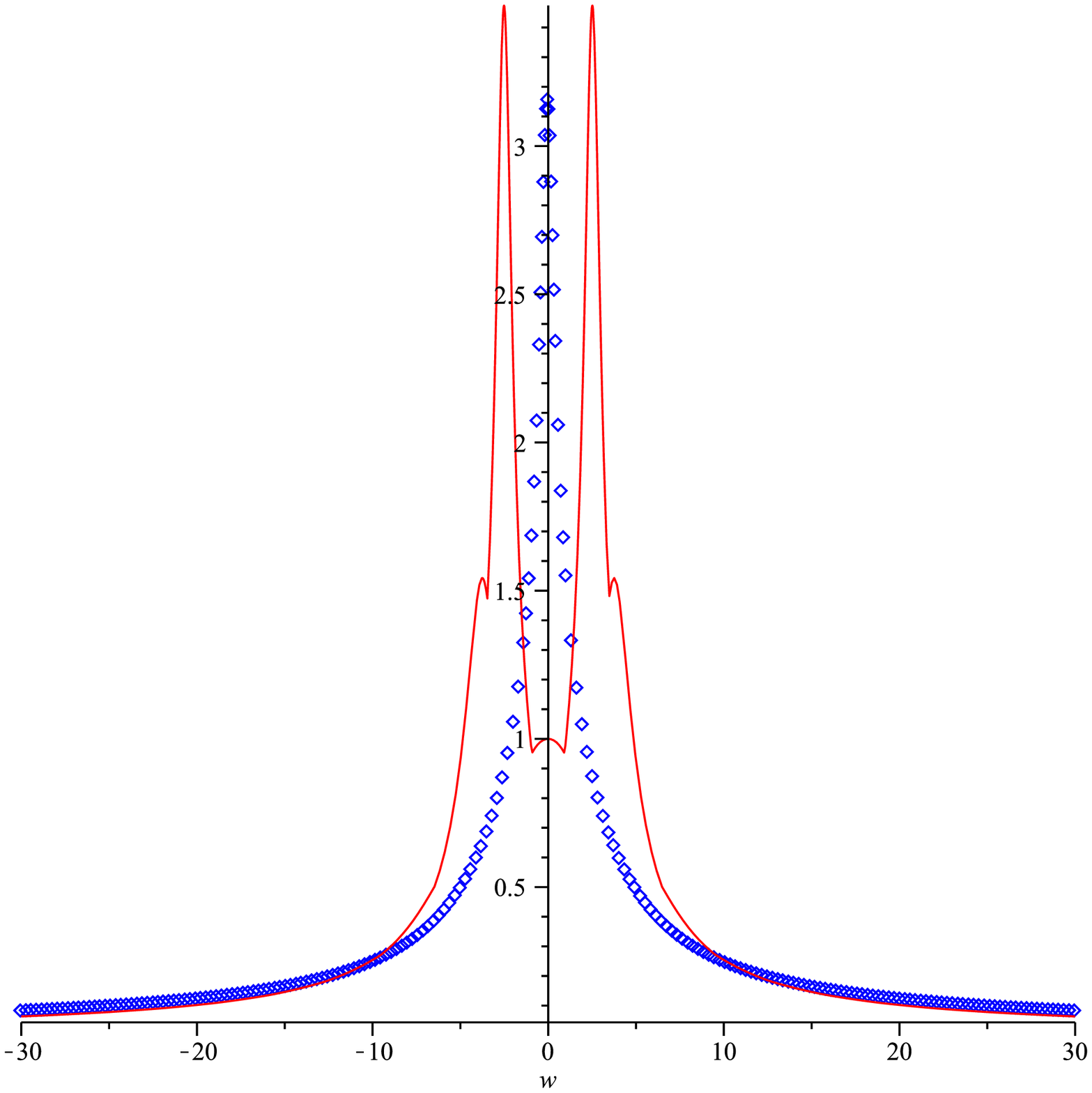}}\subfigure[]{
\includegraphics[width=4cm,height=5cm]{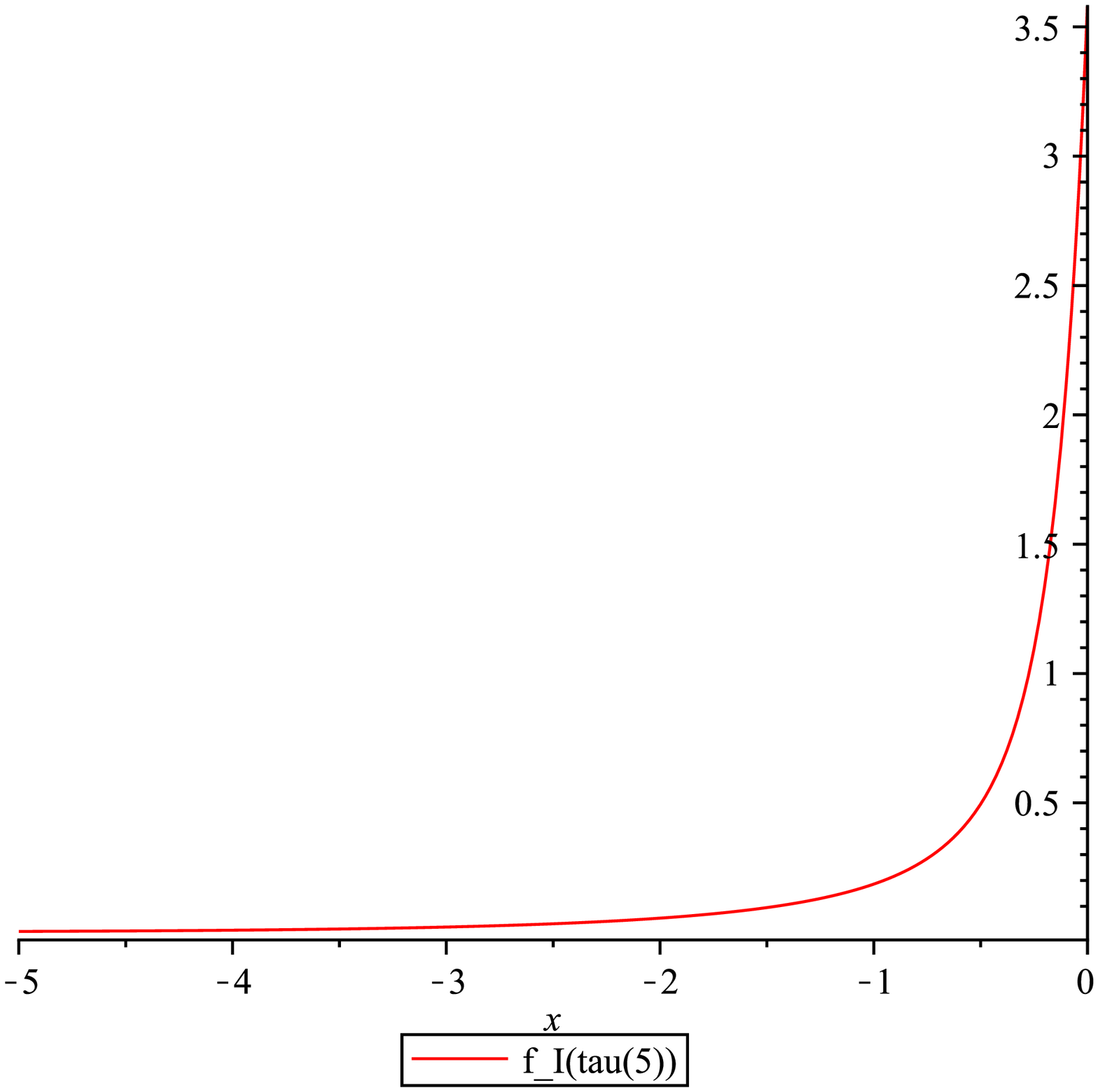}}\subfigure[]{
\includegraphics[width=4cm,height=5cm]{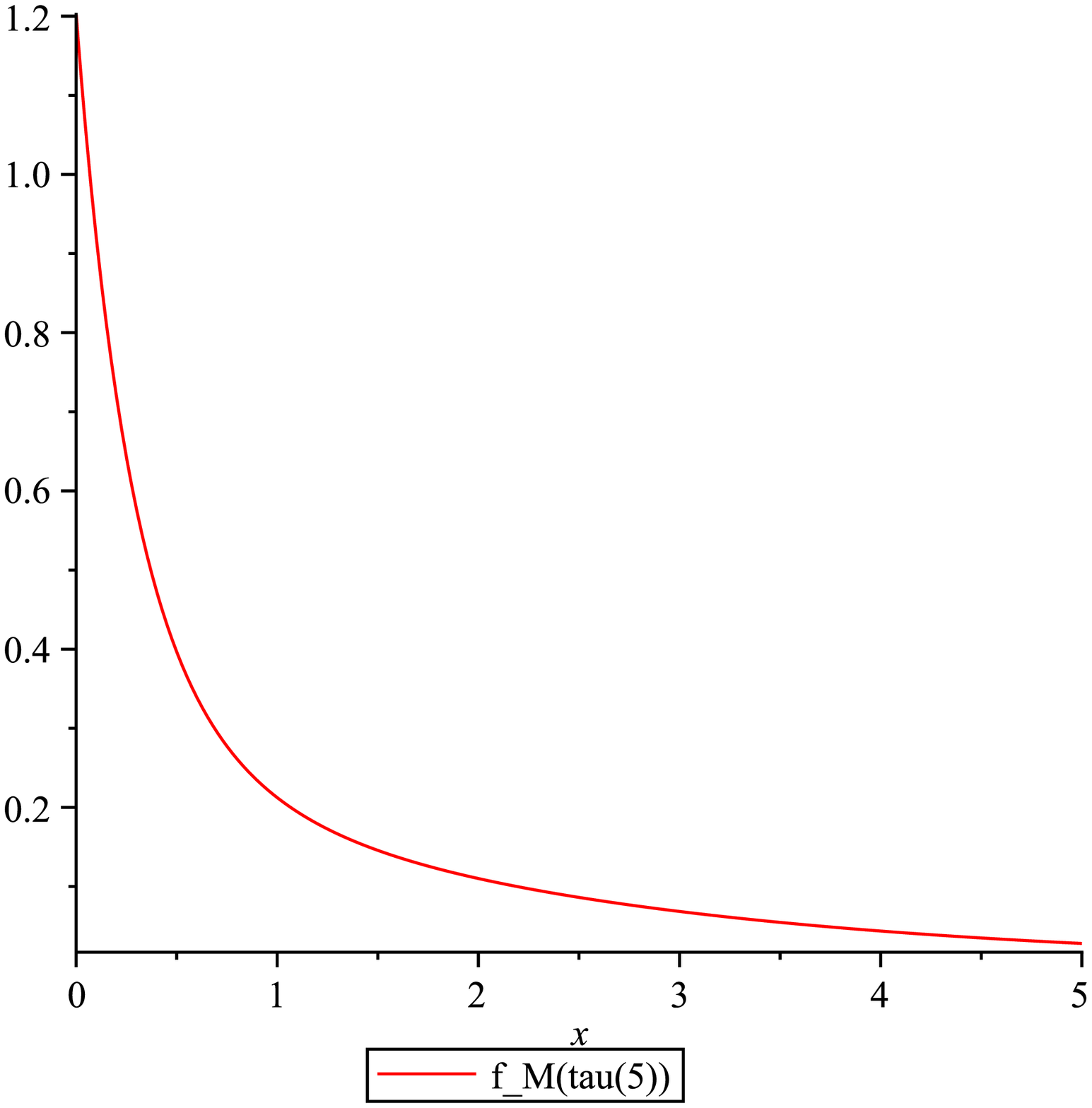}}
\caption{\scriptsize Graphical illustration of: (a)
$\frac{5}{5-\psi(\omega)}$ and its approximation $r(\omega)$; (b)
$f_{I_{\tau(5)}};$ and (c) $f_{M_{\tau(5)}}.$}
\end{figure}
\end{center}

Figures 3-b and 3-c illustrate behavior of
$f_{I_{\tau(5)}}(\cdot)$ and $f_{M_{\tau(5)}}(\cdot),$
respectively. Since the generalized hyperbolic process has
completely monotone jump density. Using \cite{Rogers}'s findings,
one may conclude that the extrema's density functions should be
completely monotone functions which cannot observe from Figures
3-b and 3-c. Such inconsistency may be interpreted by the fact
that approximations of a completely monotone function is not
completely monotone. On the other hand, since, we have $L^2({\Bbb
R})$ norm approximation. Then, our approximation should be closed,
in $L^2({\Bbb R})$ sense, to some completely monotone functions in
${\Bbb R}.$ In general, small oscillations are not a big problem,
but we hope not to see functions that look like $xsin(x),$ for
example, with increasingly large oscillations.
\end{example}

\begin{example}
\label{Ruin-generalized-hyperbolic-processes} Suppose $X_t$ in the
surplus process \eqref{surplus_process} is a generalized
hyperbolic process, given by Example
\eqref{generalized-hyperbolic-processes}. Moreover, suppose that
the random stoping time $\tau(q)$ has an exponential distribution
with mean 0.2. Using result of Example
\eqref{generalized-hyperbolic-processes}, Figure 4 illustrates
behavior of the finite-time ruin probability for different initial
value $u.$
\begin{figure}[h!]
\centering
\includegraphics[width=7cm,height=5cm]{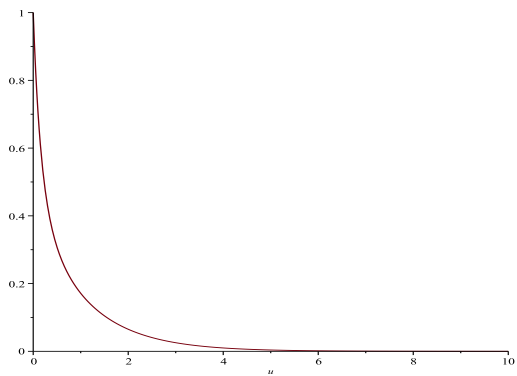}
\caption{\scriptsize Behavior of the finite-time ruin probability
for different initial value $u.$}
\end{figure}
\end{example}
\section{Conclusion and suggestion}
This article considers approximately the extrema's density
functions of a class of L\'evy processes. It provides two
approximation techniques for approximating such the density
functions. Namely, it suggests to replace $q/(q-\psi(\cdot))$ (or
$(1-q)/(1-q\exp\{-\psi(\cdot))\}$) by a sequence of
positive-definite rational functions. Two practical approximation
procedures along several examples are given. The methods presented
in this article can be generalized to other situations where the
multiplicative WHF is applicable, such as finding first/last
passage time and the overshoot, the last time the extrema was
archived, several kind of option pricing, etc. Using
\cite{Payandeh-Kucerovsky2014}'s findings, result of this article
may be generalized to a class of multivariate L\'evy processes.
\section*{Acknowledgements}
The support of Natural Sciences and Engineering Research Council
(NSERC) of Canada are gratefully acknowledged by Kucerovsky. This
article has been reviewed and comments by several authors. Hereby,
we would like appreciate for their constrictive comments. Our
special thank goes to professors Kuznetsov, Lewis, and Mordecki.
Useful comments and suggestions of anonymous reviewer is highly
appreciated.
%%%
%%%
%%%
%\section{\bf References}\label{refs}

\end{document}